\documentclass[11pt]{amsart}
\usepackage{amsmath, amsthm, amsfonts, amssymb}
\usepackage{mathrsfs,graphicx}
\providecommand{\noopsort[1]{}}
\usepackage{bbm}
\usepackage{dsfont}
\usepackage{ifthen}
\usepackage{mathbbol}
\numberwithin{equation}{section}
\usepackage{a4}
\usepackage{hyperref}
\usepackage{enumerate}
\usepackage{mathtools}

\newtheorem{thm}{Theorem}[section]
\newtheorem{cor}[thm]{Corollary}
\newtheorem{prop}[thm]{Proposition}
\newtheorem{lem}[thm]{Lemma}

\theoremstyle{definition}
\newtheorem{rem}[thm]{Remark}

\newtheorem{defn}[thm]{Definition}
\newtheorem{hyp}[thm]{Hypothesis}

\newcommand{\coloneqq}{\mathrel{\mathop:}=}

\newcommand{\one}{\mathds{1}}

\newcommand{\R}{\mathds{R}}

\newcommand{\K}{\mathds{K}}
\newcommand{\N}{\mathds{N}}

\newcommand{\expect}{\mathds{E}}

\newcommand{\Pb}{\mathds{P}}

\newcommand{\Rd}{{\mathds{R}^d}}
\newcommand{\dx}{\,\mathrm{d}}
\newcommand{\dy}{\,{\mathrm d}y}
\newcommand{\dv}{\,{\mathrm d}v}
\newcommand{\dz}{\,{\mathrm d}z}
\newcommand{\dt}{\,{\mathrm d}t}
\newcommand{\dw}{\,{\mathrm d}w}
\newcommand{\ds}{\,{\mathrm d}s}

\newcommand{\ve}{\varepsilon}

\newcommand{\pp}{\mathbf{P}}
\newcommand{\ee}{\mathbf{E}}

\let\k\undefined
\newcommand{\k}{\mathbb{k}}
\newcommand{\x}{\mathbb{x}}
\newcommand{\y}{\mathbb{y}}
\newcommand{\z}{\mathbb{z}}
\newcommand{\X}{\mathbb{X}}
\newcommand{\Z}{\mathbb{Z}}
\newcommand{\D}{\mathbb{D}}

\let\u\undefined
\newcommand{\u}{\mathbb{u}}
\let\v\undefined
\newcommand{\v}{\mathbb{v}}
\let\K\undefined
\newcommand{\K}{\mathbb{K}}

\newcommand{\f}{\mathbb{f}}
\newcommand{\g}{\mathbb{g}}
\newcommand{\cF}{\mathscr{F}}
\newcommand{\cG}{\mathscr{G}}

\usepackage{color}

\begin{document}
\title{Stable processes with reflections}
\author[K. Bogdan]{Krzysztof Bogdan}
\address{Faculty of Pure and Applied Mathematics,
Wroc\l aw University of Science and Technology,
Wyb. Wyspia\'nskiego 27, 50-370 Wroc\l aw, Poland}
\email{krzysztof.bogdan@pwr.edu.pl}
\author[M. Kunze]{Markus Kunze}
\address{Universit\"at Konstanz, Fachbereich Mathematik und Statistik, Fach 193, 78357 Konstanz, Germany}
\email{markus.kunze@uni-konstanz.de}
\date{\today}
\thanks{K.B. was supported in part by National Science Centre (Poland), grant Opus 2023/51/B/ST1/02209.}

\begin{abstract}
We construct a Hunt process that can be described as an iso\-trop\-ic $\alpha$-stable Lévy process reflected from the complement of a bounded open Lipschitz set. In fact, we  introduce a new analytic method for concatenating Markov processes. It is based on nonlocal Schrödinger perturbations of sub-Markovian transition kernels and the construction of two supermedian functions with different growth rates at infinity. We apply this framework to describe the return distribution and the stationary distribution of the process. To handle  the strong Markov property at the reflection time, we introduce a novel ladder process, whose transition semigroup encodes 
not only 
the position of the process, but also 
the number of reflections.
\end{abstract}

\maketitle

\section{Introduction}\label{s.intro}

Let $Y\coloneqq(Y_t, t \ge 0)$ be the isotropic $\alpha$-stable Lévy process in $\R^d$, with $\alpha \in (0, 2)$ and $d \in \N$. 
Let $D \subset \R^d$ be a nonempty, open, bounded Lipschitz set. Consider
\begin{equation}
\label{e.dfet}
\tau_D \coloneqq \inf\{t > 0 : Y_t \notin D\},
\end{equation}
the first-exit time of $Y$ from $D$. Naturally, $Y_{\tau_D} \in D^c$.
We will construct a strong Markov process $X\coloneqq(X_t, t \ge 0)$ that coincides with $Y$ \textit{before} $\tau_D$, but at $\tau_D$, 
instead of 
leaving to the point $z=Y_{\tau_D}\in D^c$, it is \textit{restarted} at a point $y \in D$, distributed according to a given probability measure $\mu(z, \mathrm{d}y)$
satisfying  the following assumptions.
\begin{hyp}\label{hyp1}
\hspace{1em}
\begin{enumerate}[\upshape (i)]
\item The map \( D^c \ni z \mapsto \mu(z, A) \) is Borel measurable for every \( A \in \mathscr{B}(D) \) and \(\mu(z, \cdot)\) is a probability measure on \(\mathscr{B}(D)\) for every \( z \in D^c \).
\item There exists a compact set \( H \subset D \) such that \(\vartheta \coloneqq \inf_{z \in D^c} \mu(z, H) > 0\).
\end{enumerate}
\end{hyp}
Thus, \(\mu\) is a Markov kernel from \(D^c\) to \(D\) with respect to the Borel \(\sigma\)-algebras 
$\mathscr{B}(D)$ and $\mathscr{B}(D^c)$ on $D$ and its complement $D^c$, respectively, see  
\cite[Appendix A3]{MR958914}, which exhibits \textit{some} concentration in \(D\).

In \cite{Bogdan2023} we introduced  a  Markovian semigroup \(K\coloneqq(K(t), t > 0)\) on \(D\), corresponding to the description above.
In the present paper, we complement the analytic results of \cite{Bogdan2023} by constructing a strong Markov process, actually a Hunt process, whose transition semigroup is $K$; see Theorem \ref{t.components}.
We also give applications to the stationary measure of $K$.

\subsection*{Challenges}
We point out that the semigroup \(K\) is \emph{not} a Feller semigroup \cite[Remark 5.2]{Bogdan2023}. Moreover, the semigroup is generally neither symmetric nor even bounded on \(L^2(D)\) \cite[Remark 3.8]{Bogdan2023}. Therefore, standard analytic methods for constructing an associated strong Markov process—such as those in \cite[Theorem 17.15]{MR4226142}, \cite[Chapter 7]{MR2778606}, and \cite[Theorem 9.26]{MR3308364}—do not apply. It should also be noted
that 
restricting attention to $D$ as the state space for our process makes it difficult, if not impossible, to verify that the process indeed fits 
the description above. We overcome this problem by enlarging the 
state space to the \emph{ladder space} $\D = \N_0\times D$. In 
this way, the process $X$ is accompanied by a counting process
$N = (N_t, t\geq 0)$ that tracks the number of the reflections performed.

\subsection*{Comparison to the Brownian motion case}
Let us briefly discuss the situation where the isotropic $\alpha$-stable process is replaced by Brownian motion. Thus, Brownian motion
is reset to $D$ each time it approaches $\partial D$.
Feller \cite{feller-diffusion} was the first to study such a reflected Brownian motion in dimension one. He coined the term \textit{immediate return process} and built his work on an earlier study of associated semigroups \cite{MR47886}. The study was later extended to multiple dimensions using different techniques by various authors, including Galakhov and Skubachevskii \cite{galaskub}, Ben-Ari and Pinsky \cite{b-ap07, b-ap09}, Arendt, Kunkel, and Kunze \cite{akk16}, and Kunze \cite{kunze20}. We also mention 
Bobrowski \cite{MR4628431} who studied concatenation of Markov processes on metric graphs via \textit{exit laws}.

We point out that in the Brownian motion case, the resulting reflected process is not a Hunt process as it cannot be realized with c\`adl\`ag paths. Indeed, if we choose $D$ as the state space of the process, then at the time $\tau_D$ of reflection, the paths have left-hand limits in $\partial D$, hence outside of $D$. On the other hand, choosing $\overline{D}$ as the state space would result in \textit{instantaneous states} (or \textit{branching points}), leading to a loss of right-continuity of paths; see \cite[Section 8.2]{MR2152573}. 
Moreover, the process is not \textit{quasi-left continuous} (see, e.g., \cite[IV.7]{MR850715}), which can be seen by considering stopping times $\tau_{D_n}$ with precompact open sets $D_n$ increasing to $D$. Thus, the corresponding process is not a Hunt process, although it is possible to construct an associated \textit{right process}; see, e.g., Ben-Ari and Pinsky \cite{b-ap09}.

\subsection*{Probabilistic approach}

The reflections discussed in this paper can be viewed as a special case of  \emph{concatenation} of Markov processes, also known by various names such as \textit{piecing-out}, \textit{resurrection}, and \textit{resetting}. A general probabilistic approach to concatenation was developed by Ikeda, Nagasawa, and Watanabe \cite{MR202197, MR232439, MR238401, MR246376, MR270456}. Later, Meyer \cite{MR0415784} provided a self-contained treatment in the setting of \textit{right processes}; see also Sharpe's account in \cite[Chapter 2]{MR958914}, and a more recent treatment by Werner \cite{MR4247975}. In this approach, the concatenation
is constructed on a large probability space with an ample supply of independent identically distributed processes. Note that 
this procedure at first only yields a right process, not a Hunt process (see \cite[Theorem 2.2]{MR238401}). The Brownian motion
case discussed above provides an example where a Hunt process cannot be expected.
However, if additional assumptions are satisfied, it can be shown that the process is, in fact, a Hunt process (see \cite[Theorem 2.5]{MR238401}). This could be a possible approach in the situation discussed in the present paper, but we choose a different path and rely instead on semigroup theory.
Part of our motivation stems from the fact that the aforementioned probabilistic constructions are notorious for their heavy reliance on measurability considerations. Moreover, while a process constructed using this method clearly satisfies the description above, the probabilistic approach requires additional effort to obtain distributional information about the process (see \cite{MR246376}).

\subsection*{Our approach}
In this article, we prove  the existence of a Hunt process with transition semigroup $K$ by using 
analytical methods.
Our approach is based on (possibly nonlocal) Schr\"odinger perturbations of sub-Markovian transition kernels, 
as described by Bogdan and Sydor \cite{MR3295773}, and on the classical potential-theoretic connections between sub-Markovian semigroups 
and Hunt processes given by Bliedtner and Hansen \cite{MR850715}; see also Bogdan and Hansen 
\cite[Corollary 9.5]{hansen2023positive}.
The latter relies on the existence of two $\lambda$-supermedian functions for $K$ with \textit{different} asymptotics at $\partial D$, as established in Section~\ref{s.ef}. This is a novel \emph{analytic} approach to concatenation of Markov processes. In this approach, the transition semigroup of the process is not viewed as the result of the construction, as it is in the probabilistic approach. Instead, the semigroup is the starting point for the construction of the process.

The structure of this article is as follows. 
In Section \ref{sect.semigroup}, we lift the semigroup $K$ on $D$ to a semigroup $\K$ on the ladder space $\D$. In Section \ref{s.ef}, we find specific $\lambda$-supermedian functions for the semigroups $K$ and $\K$. In Section \ref{s.hp}, we construct the Hunt process $X$ and verify that it fits the description above. 
In Section~\ref{s.ir}, we provide an explicit formula for the stationary distribution of \(K(t)\) if \(\mu(z, \mathrm{d}y) = \mathfrak{m}(\mathrm{d}y)\), i.e., it is independent of \(z \in D^c\). In Section~\ref{s.id}, the result is extended, though less explicitly, to the case of  general \(\mu\).

\begin{rem}\label{r.nc}
Compared to \cite{Bogdan2023}, in this work we do not require the map \(D^c \ni z \mapsto \mu(z, \cdot)\) to be weakly continuous at \(\partial D\). This continuity condition  was used in \cite{Bogdan2023} to discuss the infinitesimal generator and boundary conditions in certain function spaces, but it was not necessary for constructing the semigroup and does not affect the results of the present work. For more information, see \cite[Remark 3.6]{Bogdan2023}.
\end{rem}

\subsection*{Acknowledgments} We thank Wolfhard Hansen for many discussions on the topic of the paper. We also thank Zhen-Qing Chen for comments on the case of diffusions with instantaneous return and Oleksii Kulyk and Tomasz Komorowski for references.

\section{Lifting the semigroup to the ladder space}\label{sect.semigroup}

All the sets, functions, measures and integral kernels considered below are tacitly assumed Borel.
We denote $\gamma f\coloneqq \int f\dx \gamma$ if the integral is well defined, for instance if $f\ge 0$ or if $f$ is bounded and $\gamma$ is a probability measure. Furthermore, for an integral kernel $V(x,\dx y)$, a\ measure $\gamma$, and an integrable function $f$, we let $\gamma V(A)\coloneq \int V(x,A)\gamma(\!\dx x)$ and $Vf(x)\coloneq \int f(y)V(x,\!\dx y)$. We will occasionally identify density functions with measures and kernel densities with integral kernels, for instance for the Green function $g_D$ defined below.

We next recall the construction of the  transition density $k(t, x,y)$ from  
\cite[Eq.\ (3.3)]{Bogdan2023},
using the notation of \cite{MR3295773}.
To this end, we recall some notation, notions, and identities related to the isotropic $\alpha$-stable
process $Y$.
The L\'evy kernel of the isotropic stable process $Y$ is given by
\[
\nu(x,\dy)=c_{d,\alpha}|x-y|^{-d-\alpha}\dy,\quad \mbox{where}\quad
c_{d,\alpha} \coloneqq \frac{2^\alpha \Gamma((d+\alpha)/2)}{\pi^{d/2}|\Gamma(-\alpha/2)|}.
\]
We denote the \emph{Dirichlet heat kernel} of $D$ by
\[
p_D(t,x,y),\quad t>0,\ x,y\in D,
\] 
see, e.g., 
Bogdan and Kunze \cite[Subsection 2.2]{Bogdan2023}. 
It is well known that $p_D(t,x,y)$ is jointly continuous and positive for all
$(t,x,y)\in (0,\infty)\times D\times D$.

The Dirichlet heat kernel is the transition density of 
\[
Y^D_t\coloneqq
\begin{cases}
  Y_t & \text{if } t <\tau_D, \\
      \partial & \text{if } t\ge \tau_D,
\end{cases}
\]
the \textit{stable process killed} upon leaving $D$. Here, $\partial$ is an isolated \textit{cemetery state}.
Namely, let $x\in D$, $M\in \N$,
$t_0\coloneqq 0< t_1 < t_2 < \cdots <t_M<\infty$, and  $A_1, \ldots, A_M\subset D$. Denote
$s_k \coloneqq t_k -t_{k-1}$ for $0<k\le M$.
Then,
\begin{align}
\nonumber & \pp^x(Y_{t_1}\in A_1, \ldots, Y_{t_M}\in A_M, t_M<\tau_D)
\\
 \label{e.dhmY}= & \int_{A_1} p_D(s_1, x, \dx y_1)\int_{A_2}p_D(s_2, y_1, \dx y_2) \cdots \int_{A_M} p_D(s_M, y_{M-1}, \dx y_M),
\end{align}
where $\pp^x$ is the distribution of the process starting from $x$: $\pp^x(Y_0=x)=1$.

For bounded or nonnegative functions $f$ on $D$, we let
\begin{equation}\label{e.dPD}
P_D(t)f(x):=\int_D p_D(t,x,y)f(y)\dx y,\qquad t>0,\quad x\in D. 
\end{equation}
This is called the transition semigroup of the killed process and we have
\begin{equation}\label{e.dPD2}
P_D(t)f(x)=\ee^x[f(Y_t);t<\tau_D]=\ee^x[f(Y^D_t)],
\end{equation}
provided we extend $f$ to $\partial$ by letting $f(\partial)\coloneqq 0$. Here, $\ee^x$ is the integration with respect to $\pp^x$.
With the help of $\nu$ and $p_D$, 
the joint distribution of $(\tau_D, Y_{\tau_D-}, Y_{\tau_D})$
is given
by the following \emph{Ikeda--Watanabe formula}
\begin{align}\label{eq:IW}
\pp^x[\tau_D\in I,\, Y_{\tau_D-}\in A,\, Y_{\tau_D}\in B]=
\int\limits_I \!\dx s \int\limits_{A} \!\dx v \int\limits_B\! \dx z \, p_D(s,x,v)\nu(v,z),
\end{align}
where $I\subset [0,\infty)$, $A\subset D$, $B\subset D^c$, see, e.g., Bogdan, Rosiński, Serafin, and Wojciechowski \cite[Subsection~4.2]{MR3737628}.
The \emph{harmonic kernel} is defined by
\begin{equation}
\label{eq.harmonic}
h_D(x,A)\coloneqq\pp^x(Y_{\tau_D}\in A), \quad x\in D, A\subset D^c.
\end{equation}
From \eqref{eq:IW}, it follows
that the kernel $h_D$ has a density function
\begin{align}\label{eq:Pk}
h_D(x,z)&=
\int_0^\infty \dx s \int_{D}\dx v\, p_D(s,x,v)\nu(v,z)\\
&=\int_{D}\dx v\, g_D(x,v)\nu(v,z), 
\quad x\in D,\ z\in D^c,\nonumber
\end{align}
where $g_D$ is the (density of the) \emph{Green kernel}, given by
\begin{equation}
    \label{eq.green}
    g_D(x,y) \coloneqq \int_0^\infty p_D(s,x,y)\dx s.
\end{equation}
As before, for integrable functions $f$ on $D$ and $\phi$ on $D^c$, we
let
\[
h_D \phi(x):=\int_{D^c} \phi(z)h_D(x,\dx z), \quad x\in D,
\]
\[
g_D f(x):=\int_{D} g_D(x,y) f(y)\dx y, \quad x\in D,
\]
and
\[
\mu f(z):=\int_D f(x)\mu(z,\dx x), \quad z\in D^c.
\]
Furthermore, motivated by the composition of integral kernels \cite{MR850715}, we define
\[
\nu\one_{D^c}\mu(v,B):=\int_{D^c} \nu(x,z)\mu(z,B)\dz,\quad v\in D,\; B\subset D,
\] 
an integral kernel on $D$.
Then, for $t>0$ and $x, y \in D$, we let
\[
k_0(t, x,y) \coloneqq p_D(t, x, y)
\]
and, for $n\in \N_0:=\{0,1,2\ldots\}$, 
\[
k_{n+1}(t, x, y) \coloneqq \int_0^t\dx s \int_D \dx v \int_D p_D(s, x, v)\nu\one_{D^c}\mu(v, \dx w) k_n(t-s, w, y).
\]
By induction, we also have
\begin{equation}\label{e.Df2}
k_{n+1}(t, x, y) = \int_0^t\dx s \int_D \dx v \int_D k_n(s, x, v)\nu\one_{D^c}\mu(v, \dx w) p_D(t-s, w, y).
\end{equation}
We then define
\begin{equation}\label{e.defk}
k(t, x, y)\coloneqq \sum_{n=0}^\infty k_{n}(t,x,y).
\end{equation}
In consequence, we get the following Duhamel (or perturbation)
formulas:
\begin{align}\nonumber
k(t, x, y) &= p_D(t,x,y) + \int_0^t\dx s \int_D \dx v \int_D p_D(s, x, v)\nu\one_{D^c}\mu(v, \dx w) k(t-s, w, y)\\
&= p_D(t,x,y) + \int_0^t\dx s \int_D \dx v \int_D k(s, x, v)\nu\one_{D^c}\mu(v, \dx w) p_D(t-s, w, y)\label{e.pf}
\end{align}
for $t>0,x,y\in D$.
Furthermore, by \cite[Lemma 3.2]{Bogdan2023}, we have
\begin{equation}\label{e.CKk}
\int_D \, k(t, x, z)k(s, z, y)\dz = k(t+s, x, y), \quad s,t>0,\quad x,y\in D,
\end{equation}
and, by \cite[Theorem 3.4]{Bogdan2023}, 
\begin{equation}\label{e.cons}
\int_D \, k(t, x ,z)\dz = 1, \quad t>0,\quad x\in D.
\end{equation}
So, $k$ is a transition probability density 
on $D$. As usual, we will occasionally write
$k(t,x,A)\coloneqq \int_A k(t,x,y)\dx y$, $A\subset D$, $t>0$, $x\in D$. 

Let \(B_b(D)\) and \(C_b(D)\) denote the spaces of bounded (Borel) measurable and bounded continuous real-valued functions on \(D\), respectively.
We write $K(t)$ for the operators on $B_b(D)$ associated to the transition density $k(t,x,y)$:
\begin{equation}\label{e.dKt}
K(t)f(x):=\int_D k(t,x,y)f(y)\dx y,\qquad t>0,\quad x\in D. 
\end{equation}
Similarly, for $n\in \N_0$, we define
\begin{equation}\label{e.dKtn}
K_n(t)f(x):=\int_D k_n(t,x,y)f(y)\dx y,\qquad t>0,\quad x\in D. 
\end{equation}
In view of the Chapman--Kolmogorov equations \eqref{e.CKk},
$K(t)$ form an operator semigroup on $B_b(D)$.
Furthermore, according to \cite[Lemma 2]{MR3295773} and the remark at the end of \cite[Section 3]{MR3295773}, we have the following identity.
\begin{lem}\label{le:prop:1}
For all $t,s > 0$, $x,y\in D$, and $n=0,1,\ldots$,
\begin{equation}\label{prop:1}
\sum_{m=0}^n\int_D k_m(s,x,z)k_{n-m}(t,z,y)=k_n(s+t,x,y).
\end{equation}
\end{lem}
We will interpret \eqref{prop:1} as Chapman--Kolmogorov equations, too. To this end, we define the \emph{ladder space}
\begin{equation}
    \D := \N_0\times D.
\end{equation}
Thus, an element $\x$ of $\D$ is of the form $\x = (m, x)$, where $m\in \N_0$ and $x\in D$. We endow $\D$ with the product topology. Note that the the Borel $\sigma$-algebra on $\D$ is the product of the power set of $\N_0$ and the Borel $\sigma$-algebra on $D$.

We also define the function
\[
\k : (0,\infty)\times\D\times\D\to [0,\infty)
\]
by setting
\[
\k(t, (m,x), (n, y)) \coloneqq \begin{cases}
k_{n-m}(t, x, y), & \mbox{ if } n\geq m,\\
0, & \mbox{ else.}
\end{cases}
\]

Then  $\k$ is a transition density with respect to $\dx \x$, the product of the counting measure on $\N_0$ and the Lebesgue measure on $D$. Indeed, \eqref{prop:1} are the Chapman--Kolmogorov equations: for $\x=(m,x)$, $\y=(n,y)$ and $\z= (l,z)$,
\begin{align*}
\int_{\D} \k(t, \x, \y) \k(s, \y, \z)\dx \y & = \sum_{n=m}^l \int_D k_{n-m}(t, x, y) k_{l-n}(s, y, z)\dx y\\
& = k_{l-m}(t+s, x, z) = \k(t+s, \x, \z).
\end{align*}
In passing, we also note that by \eqref{e.defk} and  \eqref{e.cons},
\begin{equation}\label{e.consbK}
\k(t, \x, \D)= 1, \quad t>0,\quad \x\in \D.
\end{equation}
In what follows, we denote functions in $B_b(\D)$ by $\f, \g$, etc., and  the components of $\f\in B_b(\D)$
by $f_n(x) \coloneqq \f (n,x)$. Then, of course, the functions $f_n \in B_b(D)$, $n\in \N_0$, are uniformly bounded. Moreover, $\f\in C_b(\D)$ if and only if $\f\in B_b(\D)$ and $f_n\in C_b(D)$ for all $n\in \N_0$. We write $\K(t)$ for the operator on $B_b(\D)$ associated with the kernel $\k(t)$, i.e.\,
\begin{equation}
    \label{e.dKKt}
    \K(t)\f(m,x)\coloneqq \sum_{n=m}^\infty K_{n-m}(t)f_n(x),\quad n\in \N_0,\; x\in D.
\end{equation}
\begin{lem}\label{l.fKKf}
For $\f\in B_b(\D)$ and $m\in \N_0$, the series on the right-hand side of \eqref{e.dKKt} converges absolutely and uniformly for $x\in D$. 
\end{lem}
\begin{proof}
   By \cite[Corollary 3.5]{Bogdan2023}, there are constants $\gamma \in (0,1)$ and $c>0$ such that $K_n(t)\one(x) \leq c\gamma^n$ for all
   $n\in \N_0$ and $x\in D$. This clearly implies the result.
\end{proof}
The reader will have no difficulty proving that, by the Chapman--Kolmogorov equations,
$\K = (\K(t), t>0)$ is a semigroup on $B_b(\D)$. We can recover $K$ from $\K$ as follows:
Given $f\in B_b(D)$, we define $\f \in B_b(\D)$ by setting $\f(m,x) = f(x)$ for all $m\in \N_0$, $x\in D$. Then,
    \begin{align}
       \K(t)\f(m,x)
        & = \sum_{n=m}^\infty\int_D k_{n-m}(t, x, y)f(y)\, \dx y =K(t)f(x).\label{eq.Xdistribution2}
    \end{align}
    
If $E$ is a Polish space (in this article, we will only use $E=D$ and $E=\D$), then a \emph{$C_b$-Feller semigroup} is a family $(T(t), t>0)$ of sub-Markovian kernel operators on $B_b(E)$ such that (i)
for $t,s>0$, the \emph{semigroup law} $T(t+s)=T(t)T(s)$ holds, (ii) $T(t)f \in C_b(E)$ for all $f\in C_b(E)$, and (iii) 
for $f\in C_b(E)$, we have $T(t)f\to f$ as $t\to 0$, uniformly on
compact subsets of $E$.
A sub-Markovian kernel operator $T$ (or its kernel) is called \emph{strongly Feller} if $Tf\in C_b(E)$ for all $f\in B_b(E)$.
A $C_b$-Feller semigroup $(T(t), t>0)$ is called \emph{strongly Feller} if $T(t)$ is strongly Feller for every $t>0$ (this strengthens the condition (ii) above).
See \cite[Appendix]{Bogdan2023} for more information. For instance, 
$P_D$ and $K$ are a $C_b$-Feller and strongly Feller 
semigroup, see \cite[Lemma 2.2 and Theorem 5.3]{Bogdan2023}. While $P_D$ is also a Feller semigroup,
we recall that $K$ is not.

\begin{prop}
    \label{p.semigroup_ladderspace} 
$\K$ is a $C_b$-Feller and strongly Feller semigroup.
\end{prop}

\begin{proof}
The Chapman--Kolmogorov equations are proved above. Let us next verify that every operator $\K(t)$ has the strong Feller property. 

As a preliminary result, we first prove that for every  $n\in \N_0$ and $t>0$, the operator $K_n(t)$ 
is strongly Feller. As $k_0 = p_D$ is strongly Feller, we only need to consider $n\in \N$. 
Let $0<s<t$ and $f\in B_b(D)$. Writing $g_s \coloneqq K_{n}(t-s)f$, it follows from Lemma \ref{le:prop:1}, that
\[
K_{n}(t)f - P_D(s)g_s = \sum_{m=1}^n K_m(s)K_{n-m}(t-s)f.
\]
Since $\|K_{n-m}(t-s)f\|_\infty\leq \|f\|_\infty$, we get
\begin{align}
& \phantom{=} |K_n(t)f(x) - P_D(s)g_s(x)|
 \leq \sum_{m=1}^n K_m(s)\|f\|_\infty(x)\nonumber\\
 & = \|f\|_\infty\sum_{m=1}^n \int_0^s\dx r\int_D\, \dx v \,p_D(s,x,v) \nu\one_{D^c}\mu(v, \dx w)k_{m-1}(s-r, w, D)\nonumber\\
 & \leq \|f\|_\infty\int_0^s\, \dx r\int_D\, \dx v \,p_D(s,x,v)\nu\one_{D^c}\mu(v, \dx w) k(s-r, w,D)\nonumber\\
 & = \|f\|_\infty \pp^x(\tau_D \leq s),\label{e.ebteet}
\end{align}
where $\tau_D$ is the first exit time \eqref{e.dfet} of the $\alpha$-stable process $Y$ from $D$
and the last equality follows from \eqref{e.cons} and the Ikeda--Watanabe formula \eqref{eq:IW}.

Note that $\pp^x(\tau_D \le s) \to 0$ as $s\to 0$, uniformly on compact subsets of $D$, see \cite[Lemma 2]{chung86}. 
As $P_D$ is strongly Feller, for every $s$, the function $P_D(s)g_s$ is continuous. 
The above shows that these functions converge locally uniformly to $K_n(t)f$ as $s\to 0$, 
so the latter function is continuous as well.

Now let $\f \in B_b(\D)$ be given. As $\N_0$ is discrete, to establish the strong Feller property for $\K(t)$, 
it suffices to prove that for fixed $m\in \N_0$,
the function $x\mapsto \K(t)\f (m, x)$ is continuous. By our preliminary result, for every $n\geq m$, the function
$x\mapsto K_{n-m}(t)f_n(x)$ is continuous. Together with \eqref{e.dKKt} and Lemma \ref{l.fKKf}, this
implies the continuity of $\K(t)\f(m,\cdot)$.
\smallskip

It remains to prove (iii) in the definition of the \(C_b\)-Feller semigroup. To that end, let \(x \in D\), \(m \in \N_0\), and observe that for \(\f \in C_b(\D)\), similarly to \eqref{e.ebteet}, 
\[
\Big|\sum_{n=m+1}^\infty K_{n-m}(t)f_n(x) \Big| \leq \|\f\|_\infty \pp^x(\tau_D \leq t).
\]
Let $t\to 0$. The right-hand side converges to $0$ uniformly on compact subsets of $D$.
Noting that $K_0(t)f_m = P_D(t)f_m
\to f_m$ uniformly on
compact subsets of $D$  because $P_D$ is $C_b$-Feller, it follows from \eqref{e.dKKt} that $\K(t)\f(m,x) \to f_m(x)$ uniformly for $x$ in compact 
subsets of $D$. As $m$ was arbitrary, this implies $\K(t)\f \to \f$ uniformly on compact subsets of $\D$.
\end{proof}

\section{Excessive functions}\label{s.ef}
Let $\delta_D(x)={\rm dist}(x,\partial D)$, $x\in \Rd$.
Our next goal is the following theorem.

\begin{thm}\label{t.ef}
Let $\lambda>0$. There is a continuous function $v\colon D \to (0, \infty)$ which is $\lambda$-excessive for $K$ and satisfies $v(x)\to \infty$ as $\delta_D(x)\to 0$.
\end{thm}

As usual, $v\ge 0$ is called \emph{$\lambda$-supermedian} for the semigroup $K$ if 
$e^{-\lambda t}K(t) v\le v $ for all $t>0$. It is called \emph{$\lambda$-excessive} if, additionally, $e^{-\lambda t}K(t) v\to v$ pointwise
as $t\to 0$, see \cite[II.3]{MR850715}.
As we shall see in the next section, Theorem~\ref{t.ef} leads to the existence of a  Hunt process corresponding to the semigroup $K$.
We will prove Theorem~\ref{t.ef} after  the following auxiliary considerations.

For $\lambda\ge 0$, $g\ge 0$ on $D^c$, and $x\in D$, we define, using \eqref{eq:IW},
\begin{align}\label{e.fug}
u^\lambda_g(x)&:=\ee^x e^{-\lambda\tau_D}g(Y_{\tau_D})=\int\limits_0^\infty \int\limits_D \int\limits_{D^c} e^{-\lambda t}p_D(t,x,y) \nu(y,z) g(z)\dx z\dx y\dt  \\
&= \int_0^\infty e^{-\lambda t} P_D(t) \nu\one_{D^c} g (x)\dt .\label{e.dulgbs}
\end{align}
For instance, for $g=\one$, the constant function on $D^c$, we get 
\begin{equation}\label{e.dul1}
u^\lambda_\one (x)=\int_0^\infty e^{-\lambda t} P_D(t) \kappa_D (x)\dt.
\end{equation} 
Here
\begin{equation}\label{e.dkD}
\kappa_D(x)\coloneqq \int_{D^c} \nu(x,z)\dx z,\quad x\in D,
\end{equation}
is called the \emph{killing intensity}.
Similarly to \eqref{e.dulgbs}, we define
\begin{align}\label{e.dvlg}
v^\lambda_g(x)&:= \int_0^\infty e^{-\lambda t} K(t) \nu\one_{D^c} g (x)\dt, \quad x\in D.
\end{align}
For instance, 
\begin{equation}\label{e.dvl1}
v^\lambda_\one (x)=\int_0^\infty e^{-\lambda t} K(t) \kappa_D (x)\dt,\quad x\in D.
\end{equation}
Of course, $v^\lambda_g\ge u^\lambda_g$. In fact, by the Duhamel formula \eqref{e.pf}, for $x\in D$,
\begin{align}\nonumber
v^\lambda_g(x)&=\int_0^\infty e^{-\lambda t} P_D(t) \nu\one_{D^c} g (x)\dt \\\nonumber
& \qquad +
\int_0^\infty e^{-\lambda t} \int_0^t K(s) \nu\one_{D^c}\mu P_D(t-s)\nu\one_{D^c} g (x)\dx s\dt \\
\nonumber
&=u^\lambda_g(x)+
\int_0^\infty e^{-\lambda s} K(s) \nu\one_{D^c}\mu  \int_s^\infty e^{-\lambda(t-s)} P_D(t-s)\nu\one_{D^c} g (x)\dt \dx s\\
\label{e.re}
&=u^\lambda_g(x)+
\int_0^\infty e^{-\lambda s} K(s) \nu\one_{D^c}\mu  u^\lambda_g (x)\dx s.
\end{align}
Similarly, by the variant of Duhamel's formula preceding \eqref{e.pf}, 
\begin{align}
v^\lambda_g(x)
&=u^\lambda_g(x)+
\int_0^\infty e^{-\lambda s} P_D(s) \nu\one_{D^c}\mu  v^\lambda_g (x)\dx s\nonumber\\
\label{e.rev}
&=u^\lambda_g(x)+
\int_0^\infty \int_D \int_{D^c} e^{-\lambda s} p_D(s,x,y)\nu(y,z) \mu v^\lambda_g(z) \dz \dy \ds.
\end{align}

\begin{lem}\label{l.h1f}
For $\lambda> 0$ and bounded $g$, $v^\lambda_g$ is finite and continuous on $D$.
\end{lem}

\begin{proof}
By \eqref{e.dvl1}, 
\begin{align}
    v_\one^\lambda(x) & = \int_0^\infty e^{-\lambda t}K(t)\kappa_D(x)\dx t\nonumber\\
    & = \sum_{n=0}^\infty \int_0^\infty e^{-\lambda t}K_n(t)\kappa_D(x)\dx t.\label{eq.sum}
\end{align}
As $K_0 = P_D$, we have $\int_0^\infty e^{-\lambda t} K_0(t) \kappa_D(x)\dx t = u_\one^\lambda(x)$ by \eqref{e.dul1}.
Since $Y$ is c\`adl\`ag, 
\begin{equation}\label{e.u1l1}
u^\lambda_\one(x)=\ee^x e^{-\lambda\tau_D}<1, \quad x\in D.
\end{equation} 
By \cite[Lemma 2.5]{Bogdan2023}, $u^\lambda_\one$ is continuous.
Thus, there is $\ve>0$, depending on $\lambda$, such that $\mu u^\lambda_\one\le 1-\ve$ on $D^c$.
Indeed, with $H$ and $\vartheta$ as in Hypothesis \ref{hyp1}, the continuity of  $u^\lambda_\one$ gives an $\ve>0$ such that
$u^\lambda_\one < 1-\ve/\vartheta$ on $H$. It follows that

\begin{align}\nonumber
\mu u^\lambda_\one(z)&=\int_D u^\lambda_\one(w)\mu(z,\dw)
=1-\int_D (1-u^\lambda_\one(w))\mu(z,\dw)\\
&\le 1- \int_H \ve/\vartheta\mu(z,\dw)= 1-\ve, \quad z\in D^c.\label{e.u1le}
\end{align}

To deal with the other summands in \eqref{eq.sum}, we use that, by \eqref{e.Df2},
\[
K_{n}(t) = \int_0^t K_{n-1}(s)\nu\one_{D^c}\mu P_D(t-s)\dx s.
\]
Then,
\begin{align*}
\phantom{v_\one^\lambda} & \phantom{=} 
\int_0^\infty e^{-\lambda t} K_n(t)\kappa_D\dt \\
& =\int_0^\infty e^{-\lambda t} \int_0^t K_{n-1}(s) \nu\one_{D^c}\mu P_D(t-s)\kappa_D\ds\dt \\
&=\int_0^\infty e^{-\lambda s} K_{n-1}(s)\nu\one_{D^c}\mu\int_s^\infty  e^{-\lambda(t-s)}P_D(t-s)\kappa_D\dt \ds\\
&=\int_0^\infty e^{-\lambda s} K_{n-1}(s)\nu\one_{D^c}\mu u^\lambda_\one \ds\\
& \leq (1-\ve)\int_0^\infty e^{-\lambda s} K_{n-1}(s)\kappa_D  \ds\le (1-\ve)^n,
\end{align*}
where we used \eqref{e.dul1}, \eqref{e.u1le}, and induction.
At this point, \eqref{eq.sum} yields
$v^\lambda_\one< \sum_{n=0}^\infty (1-\ve)^n=1/\ve<\infty$. By majorization, $v_g^\lambda$ is bounded, too. 

To prove that $v^\lambda_g$ is continuous, first note that $D\ni x\mapsto p_D(t,x,y)$ is continuous for all $t>0$, $y\in D$. 
Since $u^0_\one\equiv 1$ is continuous, 
Vitali's convergence theorem (see, e.g., \cite[Chapter 22]{MR3644418}) 
yields uniform integrability of the integrand in \eqref{e.fug} for $g=\one$ and $\lambda=0$. 
As $\mu v^\lambda_g$ is bounded, by majorization, the integrand of \eqref{e.rev} is uniformly integrable, too, and the continuity of $v_g^\lambda$ follows from Vitali's theorem. 
\end{proof}

For $r>0$, we denote
\begin{align*}
\Pi_r:=\{x\in D^c:\ \delta_D(x)\le r\}.
\end{align*}
\begin{lem}
    \label{l.eshgc}
For $\lambda>0$, $v^\lambda_{\one_{\Pi_r}}\to 0$  as $r\to 0$, locally uniformly on $D$. 
\end{lem}
\begin{proof}
    As $v^\lambda_{\one_{\Pi_r}}\le v^\lambda_\one$, dominated convergence, Lemma~\ref{l.h1f}, 
    and Dini's theorem yield $v^\lambda_{\one_{\Pi_r}}\to 0$  uniformly on compact subsets of $D$.
\end{proof}

\begin{proof}[Proof of {\rm Theorem~\ref{t.ef}}]
Pick an increasing sequence $D_n$ of compact subsets of $D$ with $\bigcup D_n = D$. By Lemma \ref{l.eshgc},
we may pick $r_n>0$ such that $v^\lambda_{\one_{\Pi_{r_n}}} \leq 2^{-n}$ on $D_n$. We may also assume
that the sequence $r_n$ is decreasing. Putting $g:=\sum_{n=1}^\infty \one_{\Pi_{r_n}}$, we get from
monotone convergence that $v_g^\lambda = \sum_{n=1}^\infty v_{\one_{\Pi_{r_n}}}^\lambda$. By Lemma \ref{l.h1f}, the summands
of the latter sum are finite and continuous. By construction, the sum converges uniformly on compact subsets of $D$, 
whence $v_g^\lambda$ itself is a finite and continuous function on $D$. 
Recall that $v^\lambda_g\ge u^\lambda_g$.
By regularity of $D$ for the Dirichlet problem, for every $r>0$, we have 
$$\liminf_{D\ni x\to \partial D} u^\lambda_{\one_{\Pi_r}}(x)=1,$$
see, e.g., \cite[Lemma 2.5]{Bogdan2023}.
So, $v^\lambda_g(x)\to \infty$ as $D\ni x\to \partial D$.
By \eqref{e.dvlg},
\begin{align*}
e^{-\lambda t}K(t) v^\lambda_g&=\int_0^\infty e^{-\lambda (s+t)} K(t+s) \nu\one_{D^c} g\ds=\int_t^\infty e^{-\lambda s} K(s) \nu\one_{D^c} g\ds\le v^\lambda_g
\end{align*}
for every $t\ge 0$, hence $v^\lambda_g$ is $\lambda$-supermedian.  
Let $x\in D$ and $0\le \phi \le 1$ be a continuous function compactly supported in $D$ with $\phi(x)=1$. It follows that
\[
e^{-\lambda t} K(t) (\phi\cdot v_g^\lambda)(x) \leq e^{-\lambda t}K(t)v_g^\lambda(x) \leq v_g^\lambda(x).
\]
Since $\phi\cdot v_g^\lambda \in C_c(D)$ and $K$ is a $C_b$-Feller semigroup,
upon letting $t\to 0$,
the left-hand side converges to $\phi(x)v_g^\lambda(x) = v_g^\lambda(x)$, which is the term on the right-hand side, proving
that $e^{-\lambda t}K(t)v_g^\lambda(x) \to v_g^\lambda(x)$. As
$x$ was arbitrary, it follows that $v_g^\lambda$ is indeed $\lambda$-excessive.
\end{proof}

\begin{rem}    
Note that Theorem~\ref{t.ef} does not generalize to the Brownian motion. Indeed, suppose that in that situation
we are given a function $v\colon D\to (0,\infty)$ which is $\lambda$-excessive for the corresponding semigroup $K$ and such that $v(x) \to \infty$ as $\delta_D(x)\to 0$. Let $D_n$, $n\in \N$, be a sequence of precompact open sets increasing
to $D$ and set 
\[
M_n\coloneqq \inf \{v(x) : x\in D\setminus D_n\}. 
\]
By our assumption on $v$, $M_n\to \infty$ as $n\to \infty$.
By domination, $v$ would also be $\lambda$-supermedian for the transition 
semigroup of the Brownian motion killed outside $D$. 
At this point, \cite[Chapter II, Proposition 2.8]{MR0264757} yields for $x\in D$, 
    \begin{align*}
    v(x)&\ge \ee^x e^{-\lambda \tau_{D_n}} v(Y_{\tau_{D_n}})
    \ge M_n \ee^x e^{-\lambda \tau_{D}} \to \infty,
    \end{align*}
    as $n\to \infty$, a contradiction.    Here, of course, $Y$ denotes the Brownian motion. Notably, the second inequality above does not hold for the $\alpha$-stable process.
  \end{rem}

Given an supermedian/excessive function for $K$, we can construct a supermedian/excessive function for
the semigroup $\K$ on the ladder space $\D$.

\begin{lem}\label{l.hhat}
Let $\lambda \geq 0$ and $\alpha\in (0,1]$. If $h$ is $\lambda$-supermedian (excessive) for $K$ then the function
$\mathbb{h}(m,x) \coloneqq \alpha^mh(x)$ is $\lambda$-supermedian (excessive) for $\K$.
\end{lem}

\begin{proof}
    Let $x\in D$, $m\in \N_0$. To prove that $\mathbb{h}$ is $\lambda$-supermedian, we note that
    \begin{align*}
     	\K(t)\mathbb{h}(m,x) & = \sum_{n=m}^\infty\alpha^n K_{n-m}(t)h(x) \leq \alpha^m\sum_{n=m}^\infty K_{n-m}(t)h(x)\\
     	 & = \alpha^m K(t)h(x) \leq \alpha^m e^{\lambda t}h(x) = e^{\lambda t}\mathbb{h}(m,x).
	\end{align*}
    If $h$ is even $\lambda$-excessive for $K$ then we let $\phi$ be a function compactly supported in $D$ with $0\leq \phi\leq 1$ and $\phi(x) = 1$, and we get
    \[
    \mathbb{h}(m,x) = \lim_{t\to 0} e^{-\lambda t} \alpha^n P_D(t)(\phi\cdot h)(x) \leq \lim_{t\to 0} e^{-\lambda t}\K(t)(\phi\cdot \mathbb{h})(m,x)
    \leq \mathbb{h}(m,x).
    \]
    This ends the proof.
\end{proof}

\section{The Hunt process}\label{s.hp}
Recall that a Hunt process is a c\`adl\`ag quasi-left continuous strong Markov process with a right continuous filtration, see, e.g., \cite[IV.7]{MR850715}. Here is the main result of this section.

\begin{thm}\label{t.existence}
There is a Hunt process $\X\coloneqq(\X_t, t\geq 0)$ on $\D$ with transition semigroup $\K$ and infinite lifetime.
\end{thm}

\begin{proof}
    Existence of this process follows from \cite[Corollary 9.5]{hansen2023positive} (see also \cite[IV.7.4 and IV.7.5]{MR850715}). Indeed, the semigroup $\K$ is right continuous in the sense of \cite[Corollary 9.5]{hansen2023positive} because it is
    $C_b$-Feller. Then, the resolvent of the transition semigroup is strongly Feller as the semigroup
    itself is strongly Feller, see Proposition \ref{p.semigroup_ladderspace}. 
    Finally, for  $\lambda>0$,  there exist strictly positive $\lambda$-supermedian functions
    $\u,\v$ such that $\u/\v\in C_0(\D)$, the space of continuous functions vanishing at infinity. Note that in our situation
    this means $\u(m,x)/\v(m,x) \to 0$ if $m\to\infty$ or $x\to\partial D$. 
    Given $\alpha \in (0,1]$ and an $\lambda$-supermedian function $h$ for $K$, denote the function $\mathbb{h}$ from Lemma \ref{l.hhat} by
    $\mathbb{h}[h,\alpha]$. Now choose $\u\coloneqq\mathbb{h}[\one,\alpha]$ with $\alpha \in (0,1)$ and $\v\coloneqq \mathbb{h}[v,1]$ with the $\lambda$-supermedian function $v$ from Theorem \ref{t.ef}. By Lemma~\ref{l.hhat}, $\u$ and $\v$ are $\lambda$-supermedian.
    Writing $v_\mathrm{min}\coloneqq \min\{v(x) : x \in D\} > 0$, we get
    \[
    \frac{\u (m,x)}{\v(m,x)} = \frac{\alpha^m}{v(x)} \leq \min\Big\{\frac{\alpha^m}{v_\mathrm{min}}, \frac{1}{v(x)}\Big\} \to 0,
    \]
    as $m\to \infty$ or $x\to \partial D$.
Since $\K$ is Markovian, 
the lifetime of $\X$ is infinite, see \cite[Remark IV.7.2.1]{MR850715}.
\end{proof}

In what follows, we write $\Pb^{(m,x)}$ for the distribution of the process $\X$ starting at $(m,x)\in \D$.
We are particularly interested in the case where $m=0$. Then, we write $\Pb^x$ as a shorthand for $\Pb^{(0,x)}$.

\begin{defn}
    The processes $N\coloneqq (N_t, t\geq 0)$ and $X\coloneqq (X_t, t\geq 0)$ are defined as the components of $\X$. 
    Thus $\X_t = (N_t, X_t)$ for all $t\geq 0$.
\end{defn}

\begin{thm}\label{t.components}
$\,$
\begin{enumerate}
[\upshape (a)]
    \item The process $N$ is almost surely increasing, integer-valued and c\`adl\`ag.
    \item The process $X$ is a Hunt process with transition semigroup $K$.
\end{enumerate}
\end{thm}

\begin{proof}
Both processes are c\`adl\`ag as $\X$ is c\`adl\`ag. Moreover, $N$ has values in $\N_0$ and $X$ has values in $D$ by the definition of the state space.

    (a) Fix $0<s<t<\infty$. For $m\geq 1$,
    \begin{align*}
       \Pb^{(m,x)}(N_s\leq N_t)
       & =  \sum_{n=0}^\infty \sum_{l=n}^\infty \Pb^{(m,x)}(\X_s\in \{n\}\times D, 
        \X_t\in \{l\}\times D)\\
         & =  \sum_{n=m}^\infty\sum_{l = n}^\infty \int_D\int_D k_{n-m}(s, x,y) k_{l-n}(t-s, y,z) \dx y\dx z\\
         & = \sum_{n=m}^\infty \int_D\int_D k_{n-m}(s, x,y) k(t-s, y,z) \dx y\dx z\\
         & =  \int_D \int_D k(s, x, y) k(t-s, y, z) \dx y\dx z = 1,
    \end{align*}
    by \eqref{e.defk} and \eqref{e.cons}.\medskip

    (b) Given $f\in B_b(D)$, set $\f(m,x) = f(x)$ for $m\in \N_0$.
    Then, by \eqref{eq.Xdistribution2},
    \begin{align}
        \expect^{(m,x)}f(X_t) & = \expect^{(m,x)}\f(\X_t) = \K(t)\f(m,x) = K(t)f(x).\label{eq.Xdistribution}
    \end{align}
    In particular, the computation shows that the distribution of $X_t$ under $\Pb^{(m,x)}$ does not depend on $m\in \N_0$.
    We may thus invoke \cite[IV.7.7]{MR850715}, with $\psi(x) = (0,x)$ and the projection $\varphi(m,x) = x$ therein, to  conclude that $X$
    is a \textit{standard process}. It follows from \eqref{eq.Xdistribution} that its transition semigroup is $K$. As $K$ is
    conservative, $X$ is a Hunt process, see \cite[Remark IV.7.2.1]{MR850715}.
\end{proof}
Our next task is to express random events of interest by using the semigroup $\K$. In doing so, we will also use the strong Markov property of the Hunt process $\X$.
To this end, we consider the stopping times $\tau_n$, defined by
\[
\tau_n \coloneqq \inf\{t\ge 0 : N_t = N_0+n\}, \quad n\in \N_0.
\]
Here, as usual, $\inf \emptyset=\infty$, but we will subsequently prove that $\tau_n<\infty$ almost surely.
Of course, $\tau_0=0$. 
To simplify notation in what follows, we let $\tau:=\tau_1$. Note that, by the right-continuity of $N$, $N_{\tau_n}=N_0+n$.
\begin{lem}\label{l.il}
For every $x\in D$, $\Pb^x(\tau< \infty)=1$ and $\Pb^x(\tau_n\to \infty)=1$.    
\end{lem}
\begin{proof}
By definition, $\{\tau=\infty\} \subset \{N_t = 0\}$ for every $t>0$. Thus
\[
\Pb^x(\tau = \infty) \leq \Pb^x(N_t=0) = \Pb^x(\X_t \in \{0\}\times D) = p_D(t, x, D) \to 0
\]
as $t\to \infty$, hence $\tau<\infty$ almost surely.
To prove the second statement, we let
$\tau_\infty \coloneqq \lim_{n\to\infty}\tau_n = \sup_n \tau_n$. By monotone convergence, for every $t>0$,
$\Pb^x(\tau_\infty > t) =\lim_{n\to\infty}
\Pb^x(\tau_n>t)$.
Then, it is enough to prove that $\lim_{n\to \infty}\Pb^x(\tau_n>t)=1$ for every $t>0$, and, indeed,
\begin{align*}
\Pb^x(\tau_n>t)&=\Pb^x(\X_t\in \{0,\ldots,n-1\}\times D)\\
&=\k(t,(0,x),\{0,\ldots,n-1\}\times D)\to 
\k(t,(0,x),\N_0\times D)=1,
\end{align*}
as $n\to \infty$, see \eqref{e.consbK}.
\end{proof}

 We want to study the process
$X$ before and at time $\tau$. 

\begin{prop}\label{p.before1ststop}
Let $x\in D$, $M\in \N$,
$t_0\coloneqq 0\leq t_1 < t_2 < \cdots <t_M<\infty$, and  $A_1, \ldots, A_M, B\subset D$. Denote
$s_k \coloneqq t_k -t_{k-1}$ for $0<k\le M$.
Then,
\begin{align}
\nonumber & \Pb^x(X_{t_1}\in A_1, \ldots, X_{t_M}\in A_M, \tau> t_M, X_\tau\in B)\\
 \label{e.dhm}= & \int_{A_1} p_D(s_1, x, \dx y_1)\int_{A_2}p_D(s_2, y_1, \dx y_2) \cdots \int_{A_M} p_D(s_M, y_{M-1}, \dx y_M)\\
 \nonumber& \quad \int_{0}^\infty \dx r \int_D\dx v \int_{D^c}\dx z \, p_D(r, y_{M}, v)\nu(v, z)\mu(z, B),
\end{align}
where, for $t_1=0$, we set $p_D(0, x, \cdot) \coloneqq \delta_x(\cdot)$, the Dirac measure at $x$.
\end{prop}

\begin{proof}
Let
\[
\sigma_n \coloneqq \sum_{j=0}^\infty \Big(t_M + \frac{j+1}{n}\Big) \one_{\big(t_M+\frac{j}{n},\, t_M +\frac{j+1}{n}\big]}(\tau), \quad n\in \N.
\]
We have $\{\tau>t_M\}=\{\sigma_n>t_M\}$, $\tau\le \sigma_n< \tau+1/n$ on the set $\{\tau>t_M\}$, and $\sigma_n=0$ elsewhere, so $\sigma_n \to \one_{\{\tau>t_M\}}\tau$ almost surely. 
Since $X$ is c\`adl\`ag, for $f\in C_b(D)$ and $x\in D$, we get
\begin{align}
\nonumber&\lim_{n\to \infty}\expect^x \big[ f(X_{\sigma_n}); X_{t_1}\in A_1, \ldots, X_{t_M}\in A_M, \sigma_n > t_M\big]\\
&= \expect^x\big[f(X_\tau); X_{t_1}\in A_1, \ldots, X_{t_M}\in A_M, \tau> t_M\big].\label{e.bop}
\end{align}

Using that for $j\in \N_0$ and $n\in \N$,
\[
\Big\{ t_M +\frac{j}{n} < \tau \leq t_M+\frac{j+1}{n}\Big\} = \big\{ N_{t_M+\frac{j}{n}}=0, N_{t_M+\frac{j+1}{n}}\geq 1\big\}.
\]

It follows that for $f\in C_b(D)$, $x\in D$ and $n\in \N$,
\begin{align*}
&\phantom{=} \expect^x \big[ f(X_{\sigma_n}); X_{t_1}\in A_1, \ldots, X_{t_M}\in A_M, \sigma_n > t_M\big]\\
& = \expect^x \big[ f(X_{\sigma_n}); X_{t_1}\in A_1, \ldots, X_{t_M}\in A_M, \tau > t_M\big]\\
& =  \sum_{j=1}^\infty\expect^x \big[ f(X_{t_M +\frac{j+1}{n}}); X_{t_1}\in A_1, \ldots, X_{t_M}\in A_M, N_{t_M +\frac{j}{n}} = 0, N_{t_M+ \frac{j+1}{n}}\geq 1\big]\\
& =  \sum_{j=1}^\infty\sum_{m=1}^\infty \expect^x \big[ f(X_{t_M +\frac{j+1}{n}}); X_{t_1}\in A_1, \ldots, X_{t_M}\in A_M, N_{t_M +\frac{j}{n}} = 0, N_{t_M+ \frac{j+1}{n}}=m\big]\\
& =  \sum_{j=0}^\infty \sum_{m=1}^\infty\int_{A_1}p_D(s_1, x, \dx y_1) \int_{A_2} p_D(s_2, y_1, \dx y_2) \cdots \int_{A_M} p_D(s_M, y_{M-1}, \dx y_{M})\\
& \qquad \int_D\int_D p_D({\scriptstyle \frac{j}{n}}, y_M, \dx y)k_m({\scriptstyle \frac{1}{n}}, y, \dx w) f(w).
\end{align*}
Here we used that $ N_{t_M+\frac{j}{n}}=0$ implies $N_{t_1}=\ldots =N_{t_M} =0$ by Theorem \ref{t.components}(a), that
$\k(r, (0,x), (0,y)) = p_D(r, x, y)$ for $r>0$ and $x,y\in D$, and that $\k(r, (0,y), (m, w)) = k_m(r,y,w)$ for $r>0$, $m\in \N$, and $y,w\in D$.

Next, we note that
\begin{align*}
& \phantom{=}\sum_{j=0}^\infty \sum_{m=1}^\infty\int_D\int_D p_D({\scriptstyle \frac{j}{n}}, y_{M},\dx y) k_m({\scriptstyle \frac{1}{n}}, y, \dx w) f(w)\\
& =  \sum_{j=0}^\infty \sum_{m=1}^\infty \int_D p_D({\scriptstyle \frac{j}{n}}, y_{M}, \dx y )\int_0^\frac{1}{n}\dx s \int_D\dx v
\int_{D^c}\dx \zeta \,p_D(s, y, v)\nu(v,\zeta )\\
& \qquad \int_{D}\mu (\zeta, \dx z) \int_{D}\dx w k_{m-1}({\scriptstyle \frac{1}{n}} - s, z, w) f(w)\\
& =  \sum_{j=0}^\infty \int_0^\frac{1}{n} \dx s \int_D\dv \int_{D^c}\dx \zeta \, p_D({\scriptstyle \frac{j}{n}} +s, y_{M}, v)\nu(v,\zeta)\\
& \qquad\int_D\mu(\zeta, \dx z)\int_D \dx w k({\scriptstyle \frac{1}{n}}-s, z, w)f(w)\\
& =  \int_0^\infty \dx r \int_D\dx v \int_{D^c}\dx\zeta\, p_D(r, y_{M}, v)\nu(v, \zeta)F_n(\zeta,r),
\end{align*}
where
\[
F_n(\zeta,r) = \sum_{j=0}^\infty \one_{({\scriptstyle \frac{j}{n}, \frac{j+1}{n}}]}(r) \int_D\mu(\zeta, \dx z) K({\scriptstyle \frac{j+1}{n}}-r)f(z).
\]
Since $K(t)f(z) \to f(z)$ for $t\to 0$, we get
\[
F_n(\zeta,r) \to F(\zeta) := \int_D\mu(\zeta, \dx w)f(w)
\]
as $n\to \infty$ for all $r\in (0,\infty)$.  In view of \eqref{e.bop},
\begin{align}
& \expect^x \big[ f(X_\tau); X_{t_1}\in A_1, \ldots, X_{t_M}\in A_M, \tau > t_M\big]\notag\\
= & \int_{A_1} p_D(s_1, x, \dx y_1)\int_{A_2} p_D(s_2, y_1, \dx y_2)\cdots \int_{A_M}p_D(s_M, y_{M-1}, \dx y_{M})\label{e.eif}\\
& \qquad \int_0^\infty \dx r \int_D \dx v\int_{D^c}\dx \zeta\, p_D(r, y_{M}, v)\nu(v,\zeta)\int_D\mu(\zeta, \dx w)f(w).\notag
\end{align}
As $f\in C_b(D)$ was arbitrary, the result follows 
since both sides of \eqref{e.dhm} are finite, hence regular 
measures in $B$; see \cite[Theorem 1.2]{MR1700749}. Note that \cite[Theorem 1.2]{MR1700749} is actually formulated
for probability measures. However, choosing $f=\one$ in \eqref{e.eif}, yields equality of masses of the two measures, 
so we can renormalize.
\end{proof}

\begin{rem}
    \label{r.general}
    Proposition \ref{p.before1ststop} was formulated for the (initial) distributions of $\X_0$ in the form $\delta_0(\!\dx m)\otimes \delta_x(\!\dx x)$, but
    it also holds true for general initial distributions of the form $\delta_{m_0}(\!\dx m)\otimes\rho(\!\dx x)$, where
    $m_0\in \N_0$ and $\rho$ is a probability measure on $D$. Indeed, generalizing from $m_0=0$ to general $m_0$ is
    straightforward and to obtain general distributions $\rho$ of $X_0$, it suffices to integrate \eqref{e.dhm} over $x$ with
    respect to $\rho$.
\end{rem}

We next compute the distribution of the random variables $X_{\tau_n}$ for $n\in \N$. To that end, we use the harmonic kernel $h_D$,
see \eqref{eq.harmonic}, and write
\[
(h_D\mu)(x, A) \coloneqq \int_{D^c} h_D(x, \dx z)\mu(z, A),\quad x\in D,\; A\subset D.
\]
We then define $(h_D\mu)^{n+1}$ for $n\in \N$ by induction:
\[
(h_D\mu)^{n+1}(x, A) \coloneqq \int_{D}(h_D\mu)^n(x, \dx y)(h_D\mu)(y, A),\quad x\in D,\; A\subset D.
\]
We also let $(h_D\mu)^{0}(x,A)\coloneqq \delta_x(A)$, the the Dirac measure at $x$.

\begin{cor}\label{c.distributionsofxtau}
For all $n\in \N_0$ and $x\in D$, we have $\Pb^x(\tau_n< \infty)=1$ and
\(
\Pb^x(X_{\tau_n} \in B) = (h_D\mu)^n(x, B)\) for every $B\subset D$.
\end{cor}

\begin{proof}
The case $n=0$ is trivial.
The case $n=1$ follows from Lemma \ref{l.il} and Proposition \ref{p.before1ststop}, applied with $M=1$, $t_1=0$, and $A_1=D$; see \eqref{eq:Pk}. To prove the general case, we use
the strong Markov property and proceed by induction. So assume that under all $\Pb^x$, $x\in D$, the stopping time $\tau_n$ is a.s.\ finite and 
that $X_{\tau_n}$ has distribution $(h_D\mu)^n(x, \cdot)$.
Then, for the random shift $\Theta_{\tau_n}$, we have $\X_{\tau_n+t} =\X_t\circ\Theta_{\tau_n}$,  see \cite[Section IV.6]{MR850715}.
Recall that, by the right-continuity of the paths, $N_{\tau_n}=N_0+n$.
Then,
\begin{align}
    \tau\circ\Theta_{\tau_n} & = \inf \{t\geq 0 : N_t\circ\Theta_{\tau_n} = N_0\circ\Theta_{\tau_n}+1\}\nonumber\\
    & = \inf \{t\geq 0 : N_{\tau_n+t} = N_{\tau_n}+1\}\nonumber\\
    & = \inf \{t+\tau_n\geq \tau_n : N_{\tau_n+t} = N_{\tau_n}+1\} -\tau_n\nonumber\\
    & = \inf \{s\geq 0 : N_s = N_0 + (n+1)\} - \tau_n = \tau_{n+1}- \tau_n,\label{eq.shifted}
\end{align}
implying that
$\tau_{n+1} = \tau\circ\Theta_{\tau_n}+\tau_n$ is finite a.s. Since
$X_{\tau_{n+1}} = X_\tau\circ \Theta_{\tau_n}$,
it follows from
the strong Markov property, that
\begin{align*}
\Pb^x(X_{\tau_{n+1}}\in B) & = \expect^x \big[\Pb^{X_{\tau_n}}(X_\tau \in B)\big]\\
&= \int_D (h_D\mu)^n(x, \dx y) (h_D\mu)(y, B) 
= (h_D\mu)^{n+1}(x, B).\qedhere
\end{align*}
\end{proof}

\begin{defn}
\label{d.killed}
    Let $\hat{D} = D\cup\{\dagger\}$, where $\dagger$ is an isolated cemetery state. For $n\in \N_0$
    we define $X^{(n)} \coloneqq (X^{(n)}_t, t\geq 0)$ by setting 
    $X^{(n)}_t \coloneqq X_{\tau_n+t}$ if $0\leq t < \tau_{n+1}- \tau_n$ and $X^{(n)}_t = \dagger$ if $t \geq \tau_{n+1}-\tau_n$. 
\end{defn}
In particular, $X^{(0)}_0=X_{0}$ and  $X^{(n)}_0=X_{\tau_n}$ for $n\in \N$.
\begin{cor}\label{c.killedprocess}
For $n\in \N_0$, the process $X^{(n)}$ is an isotropic $\alpha$-stable L\'evy process killed upon exiting $D$.
If $x\in D$ and $n\in \N$, then under $\Pb^x$, the initial distribution of $X^{(n)}$ is  $(h_D\mu)^{n}(x, \cdot)$.
\end{cor}
\begin{proof}
It follows from Proposition \ref{p.before1ststop} with $B=D$, that the finite-dimensional distributions of $X^{(0)}$
are that of the killed isotropic $\alpha$-stable L\'evy process starting from $x$. This proves the claim for $n=0$. 
For $n\geq 1$, the strong Markov property 
implies that the process $\Z =(\Z_t, t\geq 0)$ defined by $Z_t\coloneqq \X_{\tau_n+t}$ is also a 
Markov process with transition semigroup $\K$. We let $\Z_t\eqqcolon (M_t,Z_t)$ and 
\[
    Z^{(0)}_t \coloneqq \begin{cases}
        Z_t, & \mbox{ if } 0\leq t < \inf\{ t\geq 0 : M_t \neq M_0\},\\
        \dagger, & \mbox{ else }.
    \end{cases}
\]

Taking Remark \ref{r.general} into account, it follows from Proposition \ref{p.before1ststop} applied to $\Z$ that
$Z^{(0)} \coloneqq (Z^{(0)}_t, t\geq 0)$ is an isotropic $\alpha$-stable L\'evy process killed upon exiting $D$. However, using
\eqref{eq.shifted} it follows that $Z^{(0)}_t = X^{(n)}_t$ for all $t\geq 0$. By Corollary \ref{c.distributionsofxtau}, the distribution of $X^{(n)}_0$ is
$(h_D\mu)^{n}(x, \cdot)$ for $n\ge 1$.
\end{proof}

\begin{rem}\label{r.nc2}
    The proof of Corollary~\ref{c.killedprocess} shows that $(X^{(0)},\tau)$ and $(Y^D,\tau_D)$ have the same distribution with respect to every $\Pb^x$, $x\in D$.
Of course, $X^{(0)}_t$ coincides with $X_t$ for  $t<\tau$. Furthermore, $X_\tau$ has the same distribution as $\mu(Y_{\tau_D}, \dx y)$,
which is $h_D\mu(x,\cdot)$ if the starting point is $x\in D$, see \eqref{eq.harmonic}.
The observations make justice to the heuristic description of $X$ in the introduction. 
Note that we may realize \( Y_t \) for \( t < \tau_D \) on the same probability space as \( X \) by letting \textit{a posteriori} \( Y_t \coloneqq X^{(0)}_t \). However, we do not claim that the distribution of \( X_{\tau} \) \textit{conditional} on \( Y_{\tau_D} \) is \( \mu(Y_{\tau_D}, \dx y) \). In fact, we do not realize \( Y_t \) for \( t \ge \tau_D \) on the same probability space as $X$. Such questions are left for the future.
\end{rem}

\section{Independent reflections}\label{s.ir}
Let $\mathfrak{m}$ be a probability measure on $D$. In this section we assume  
$$\mu(w,\dz):=\mathfrak{m}(\!\dz),\quad w\in D^c.$$ 
Of course, $\mu$ satisfies Hypothesis~\ref{hyp1}. 
We could say, somewhat ambiguously, that $\mu(w,\dz)$ is \textit{independent} of $w$. By Corollary \ref{c.distributionsofxtau},
$X_\tau$ has distribution $\mathfrak{m}$.
In what follows, we write $\Pb^\mathfrak{m}$ for the distribution of the process $\X$ starting from the distribution $\delta_0(\!\dx m)\mathfrak{m}(\!\dx x)$.
So, under $\Pb^\mathfrak{m}$, the  distribution of $X_0$ is $\mathfrak{m}$ and $\Pb^\mathfrak{m}(N_0=0)=1$.

Denote by $(\cF_t : t\geq 0)$ the natural filtration generated by $\X$. We put
\[
\cF_\infty \coloneqq \bigvee_{t\geq 0} \cF_t
\]
and define some related $\sigma$-algebras. The $\sigma$-algebra $\cF_{\tau-}$ is defined by
\[
\cF_{\tau-}\coloneqq \sigma ( \{t<\tau\}\cap \Lambda : \Lambda \in \cF_t, t \geq 0\},
\]
see \cite[Section 1.3]{MR2152573}. We note that this is different from $\cF_\tau$, the usual $\sigma$-algebra appearing in the strong Markov property. As $\cF_t = \sigma(\X_s : s \leq t)$
and $N_s = 0$ for every $s<\tau$, it follows that $\cF_{\tau-}$ is generated by sets of the form
\[
\{ \X_s \in \{0\}\times A \}, \quad s\geq 0, A \subset D.
\]
We also define
\[
\cG_\tau \coloneqq \sigma (\X_{\tau+t} : t\geq 0).
\]

Making use of the shift $\Theta_\tau$,
we get $\cG_\tau = \Theta_\tau\cF_\infty$.

\begin{lem}\label{l.independence}
\begin{enumerate}
[\upshape (a)]
\item Under every measure $\Pb^x$, the $\sigma$-algebras $\mathscr{F}_{\tau-}$ and $\sigma(X_\tau)$ are independent
and for $\Lambda\in \mathscr{F}_{\tau-}$ and $B\subset D$, we have
\[
\Pb^x(\Lambda\cap\{X_\tau \in B\}) = \Pb^x(\Lambda)\mathfrak{m}(B).
\]
\item Under the measure $\Pb^\mathfrak{m}$, the $\sigma$-algebras $\mathscr{F}_{\tau-}$ and $\mathscr{G}_\tau$ are independent.
\end{enumerate}
\end{lem}
\begin{proof}
(a) For  $0\leq t_1 < t_2 <\cdots < t_N$ and $B,A_1, \ldots A_N \subset D$, define 
\[
\Lambda\coloneqq \{ X_{t_1}\in A_1, \ldots, X_{t_N} \in A_n, \tau > t_{N}\}.
\]
Then $\Lambda\in \mathscr{F}_{\tau-}$.
It follows from Proposition \ref{p.before1ststop} and the special structure of $\mu(z, \cdot)\equiv \mathfrak{m}(\cdot)$,
that $\Pb^x(\Lambda\cap\{X_\tau \in B\}) = \Pb^x(\Lambda)\mathfrak{m}(B)$.
Since events of the form $\{X_{t_1}\in A_1, \ldots, X_{t_N}\in A_N, \tau> t_N\}$ generate the $\sigma$-algebra $\mathscr{F}_{\tau-}$
and are stable under intersections, (a) is proved.\medskip

\noindent
(b) By the strong Markov property, $(\X_{\tau+t})_{t\geq 0}$ is a Markov process with transition semigroup $\K$. Thus, 
if $\Gamma=\Theta_\tau\tilde{\Gamma} \in \mathscr{G}_\tau$, where $\tilde{\Gamma} \in \mathscr{F}_\infty$, then
$\Pb^x(\Gamma) = \Pb^\rho (\tilde{\Gamma})$ for all $x\in D$. Here, $\rho$ is the distribution of $X_\tau$ under the measure $\Pb^x$. 
By part (a), $\rho=\mathfrak{m}$ and it follows that for $\Lambda\in \mathscr{F}_{\tau-}$,
\[
\Pb^{\mathfrak{m}}(\Lambda\cap \Gamma ) = \expect^\mathfrak{m}\one_\Lambda \expect^{X_\tau}\one_{\tilde{\Gamma}}
= \Pb^{\mathfrak{m}}(\Lambda) \int_D \Pb^x(\tilde{\Gamma}) \mathfrak{m}(\dx x) = \Pb^\mathfrak{m}(\Lambda) \Pb^\mathfrak{m}(\Gamma).\qedhere
\]
\end{proof}

\begin{cor}\label{c.independent}
Under $\Pb^\mathfrak{m}$, the processes $(X^{(n)})_{n\in \N_0}$ from Definition \ref{d.killed}
are independent and have the same distribution
as the killed $\alpha$-stable process with initial distribution $\mathfrak{m}$.
\end{cor}

\begin{proof}
By Corollary \ref{c.killedprocess} and \ref{c.distributionsofxtau}, every process $X^{(n)}$ is a killed $\alpha$-stable process with initial distribution $\mathfrak{m}$. The independence follows inductively from 
Lemma \ref{l.independence}(b).
\end{proof}

It was proved in \cite[Theorem 6.2]{Bogdan2023} that the semigroup $K$ has a unique stationary distribution (density) $\kappa$. 
We can now explicate $\kappa$; the following two theorems are analogues of \cite[Proposition 1]{b-ap09}, but the proofs are different.
\begin{thm}\label{t.sm}
If there is a probability measure $\mathfrak{m}$ on $D$ such that $\mu(w,\dz)=\mathfrak{m}(\!\dz)$ for all $w\in D$, then the stationary density $\kappa$ equals $c\mathfrak{m} g_D$. 
\end{thm}
\begin{proof}
We have
\[
\expect^\mathfrak{m} \tau_1=\int_D\int_D \mathfrak{m}(\!\dx x) g_D(x,y)\dy
\]
and, for every bounded function $f:D\to[0,\infty]$,
\[
\expect^\mathfrak{m} \int_0^{\tau_1} f(X_t)dt=\int_D\int_D \mathfrak{m}(\!\dx x)g_D(x,y) f(y)\dx y.
\]
By Corollary \ref{c.independent}, 
the processes $(X_{\tau_n+t}, 0\le t<\tau_{n+1}-\tau_n)$ are independent and identically distributed under $\Pb^\mathfrak{m}$.
By the strong law of large numbers and the ergodic theorem \cite[Theorem 3.3.1 and 3.2.6 and Proposition 2.1.1]{MR1417491},
it holds $\Pb^\mathfrak{m}$-almost surely that
\begin{align*}
\expect^\mathfrak{m} \int_0^{\tau_1} f(X_t)dt/\expect^\mathfrak{m} \tau_1
&= \lim_{n\to \infty}
\int_{0}^{\tau_{n}} f(X_t)\dx t/\int_{0}^{\tau_{n}}\dx t\\
&= \lim_{T\to \infty}\frac{1}{T}
\int_{0}^{T} f(X_t)\dx t=\int_Df(x)\kappa(\!\dx x).
\end{align*}
This yields the result.
Of course, $c=1/\int_D\int_D g_D(x,y)\mathfrak{m}(\dy)\dx x$. 
\end{proof}

\section{Stationary measure: general case}\label{s.id}
This section extends Theorem~\ref{t.sm} to arbitrary kernels $\mu$ considered in this paper, using a different proof.
Let $R_n\coloneqq X_{\tau_n}$, the position after the $n$-th reflection. Recall $h_D$ from \eqref{eq.harmonic}.
\begin{lem}
    \label{l.discretemc}
    The process $(R_n)_{n\in \N}$ is a discrete time Markov chain with transition kernel $h_D\mu$ and
a unique stationary measure  $\mathfrak{p}$.
\end{lem}

\begin{proof}
    That $(R_n)_{n\in \N}$ is a discrete Markov chain follows from the strong Markov property of the 
    continuous time Markov process $(\X_t)_{t\geq 0}$ and the statement about the transition probabilities follows
    from Corollary \ref{c.distributionsofxtau}. As for the existence of a stationary measure,
    we consider an auxiliary transition kernel $\mu h_D$ on  $D^c$. Since $h_D$ has a density, the same is true of $\mu h_D$. 
By Hypothesis \ref{hyp1}(ii) and the Harnack inequality \cite{MR1703823}, there is a constant $C$ such that for all $z,w\in D^c$, 
\begin{align*}
\mu h_D(z,w)&\ge 
\int_H \mu(z,\dx x) h_D(x,w)\\
&\ge
C\int_H \mu(z,\dx x) h_D(x_0,w)\ge 
C\vartheta h_D(x_0,w)>0,
\end{align*}
where $x_0$ is an arbitrary point in $H$. This \emph{minorization condition} 
yields the Dobrushin condition whence a unique stationary   distribution $\eta$ for $\mu h_D$: $\eta \mu h_D=\eta$, see \cite[Theorem 1.3 and p. 33]{kulik15}.
Then $\mathfrak{p}\coloneqq h_D \eta \mu$ is stationary for $h_D\mu$. Indeed, $h_D\eta\mu h_D\mu=h_D(\eta\mu h_D)\mu=h_D\eta\mu$. Uniqueness of the
stationary measure for $h_D\mu$ follows the same lines: The measure is stationary for $(h_D\mu)^2=h_D\mu h_D\mu$. Then, for $x\in D$, $A\subset D$,
\begin{align*}
h_D\mu h_D\mu(x,A)&\ge
\int_H h_D\mu(x,\dv) \int_{D^c} h_D(v,w)  \mu(w,A)\dx w\\
&\ge
C
\int_H h_D\mu(x,\dv) \int_{D^c} h_D(x_0,w)  \mu(w,A)\dx w\\
&\ge
C \vartheta
\int_{D^c} h_D(x_0,w) \mu(w,A)\dx w.
\end{align*}
Therefore we get the Dobrushin condition:
\[
\|h_D\mu h_D\mu(x,\cdot)-h_D\mu h_D\mu(y,\cdot)\|_{TV}\le 2(1- C\vartheta).\qedhere
\]
\end{proof}
   
    Recall that $\mathfrak{p}$ is the unique stationary distribution of the Markov chain $(R_n)_{n\in \N}$ and $\kappa$ is the unique stationary distribution  of the semigroup $K$.
\begin{thm}
    \label{t.invariantmeasure2}
 The stationary density $\kappa$ equals \(c \mathfrak{p} g_D\).
\end{thm}

\begin{proof}
   It follows from Corollary \ref{c.killedprocess} that for $h\in B_b(D)$, 
    \[
    \int_D h(y)[\mathfrak{p} g_D](y)\dx y = \expect^\mathfrak{p} \int_0^{\tau} h(X_s)\dx s.
    \]
    Applying this to 
    \[
    h(y) = (K(t)f)(y) = \expect^y f(X_t)
    \]
    for $f\in B_b(D)$, we get
    \begin{align*}
        \langle K(t)f, \mathfrak{p} g_D\rangle & = \expect^\mathfrak{p} \int_0^{\tau} \expect^{X_s} f(X_t)\dx s = \expect^\mathfrak{p} \int_0^\tau f(X_{t+s})\dx s\\
        & = \expect^\mathfrak{p} \int_t^{\tau+t} f(X_s)\dx s\\
        & = \expect^\mathfrak{p} \int_0^\tau f(X_s)\dx s + \expect^\mathfrak{p}\int_{\tau}^{\tau+t}f(X_s)\dx s - \expect^\mathfrak{p}\int_0^t f(X_s)\dx s\\
        & = \expect^\mathfrak{p}\int_0^\tau f(X_s)\dx s = \langle f, \mathfrak{p} g_D\rangle.
    \end{align*}
    Here, the second equality follows from the Markov property and the fifth equality from the fact that, by the choice of $\mathfrak{p}$ as the starting distribution, 
    the processes $(X_s)_{s\geq 0}$ and $(X_{\tau+s})_{s\geq 0}$ have the same initial distribution and the same transition probabilities, whence
    \[
    \expect^\mathfrak{p}\int_{\tau}^{\tau+t}f(X_s)\dx s = \expect^\mathfrak{p}\int_0^t f(X_s)\dx s    
    \]
Of course,   we take  $c=1/\int_D\int_D g_D(x,y)\mathfrak{p}(\!\dx x) \dx y$, to normalize $\mathfrak{p} g_D$.
\end{proof}


\begin{thebibliography}{10}

\bibitem{akk16}
W.~Arendt, S.~Kunkel, and M.~Kunze.
\newblock Diffusion with nonlocal boundary conditions.
\newblock {\em J. Funct. Anal.}, 270(7):2483--2507, 2016.

\bibitem{b-ap07}
I.~Ben-Ari and R.~G. Pinsky.
\newblock Spectral analysis of a family of second-order elliptic operators with
  nonlocal boundary condition indexed by a probability measure.
\newblock {\em J. Funct. Anal.}, 251(1):122--140, 2007.

\bibitem{b-ap09}
I.~Ben-Ari and R.~G. Pinsky.
\newblock Ergodic behavior of diffusions with random jumps from the boundary.
\newblock {\em Stochastic Process. Appl.}, 119(3):864--881, 2009.

\bibitem{MR1700749}
P.~Billingsley.
\newblock {\em Convergence of probability measures}.
\newblock Wiley Series in Probability and Statistics: Probability and
  Statistics. John Wiley \& Sons, Inc., New York, second edition, 1999.
\newblock A Wiley-Interscience Publication.

\bibitem{MR850715}
J.~Bliedtner and W.~Hansen.
\newblock {\em Potential theory}.
\newblock Universitext. Springer-Verlag, Berlin, 1986.
\newblock An analytic and probabilistic approach to balayage.

\bibitem{MR0264757}
R.~M. Blumenthal and R.~K. Getoor.
\newblock {\em Markov processes and potential theory}.
\newblock Pure and Applied Mathematics, Vol. 29. Academic Press, New
  York-London, 1968.

\bibitem{MR4628431}
A.~Bobrowski.
\newblock Concatenation of nonhonest {F}eller processes, exit laws, and limit
  theorems on graphs.
\newblock {\em SIAM J. Math. Anal.}, 55(4):3457--3508, 2023.

\bibitem{MR1703823}
K.~Bogdan and T.~Byczkowski.
\newblock Probabilistic proof of boundary {H}arnack principle for
  {$\alpha$}-harmonic functions.
\newblock {\em Potential Anal.}, 11(2):135--156, 1999.

\bibitem{Bogdan2023}
K.~Bogdan and M.~Kunze.
\newblock The fractional {Laplacian} with reflections.
\newblock {\em Potential Anal.}, 61:317--345, 2024.

\bibitem{MR3737628}
K.~Bogdan, J.~Rosi\'{n}ski, G.~Serafin, and {\L}.~Wojciechowski.
\newblock L\'{e}vy systems and moment formulas for mixed {P}oisson integrals.
\newblock In {\em Stochastic analysis and related topics}, volume~72 of {\em
  Progr. Probab.}, pages 139--164. Birkh\"{a}user/Springer, Cham, 2017.

\bibitem{MR3295773}
K.~Bogdan and S.~Sydor.
\newblock On nonlocal perturbations of integral kernels.
\newblock In {\em Semigroups of operators---theory and applications}, volume
  113 of {\em Springer Proc. Math. Stat.}, pages 27--42. Springer, Cham, 2015.

\bibitem{chung86}
K.~L. Chung.
\newblock Doubly-{F}eller process with multiplicative functional.
\newblock In {\em Seminar on stochastic processes, 1985 ({G}ainesville, {F}la.,
  1985)}, volume~12 of {\em Progr. Probab. Statist.}, pages 63--78.
  Birkh\"{a}user Boston, Boston, MA, 1986.

\bibitem{MR2152573}
K.~L. Chung and J.~B. Walsh.
\newblock {\em Markov processes, {B}rownian motion, and time symmetry}, volume
  249 of {\em Grundlehren der mathematischen Wissenschaften [Fundamental
  Principles of Mathematical Sciences]}.
\newblock Springer, New York, second edition, 2005.

\bibitem{MR1417491}
G.~Da~Prato and J.~Zabczyk.
\newblock {\em Ergodicity for infinite-dimensional systems}, volume 229 of {\em
  London Mathematical Society Lecture Note Series}.
\newblock Cambridge University Press, Cambridge, 1996.

\bibitem{MR47886}
W.~Feller.
\newblock The parabolic differential equations and the associated semi-groups
  of transformations.
\newblock {\em Ann. of Math. (2)}, 55:468--519, 1952.

\bibitem{feller-diffusion}
W.~Feller.
\newblock Diffusion processes in one dimension.
\newblock {\em Trans. Amer. Math. Soc.}, 77:1--31, 1954.

\bibitem{MR2778606}
M.~Fukushima, Y.~Oshima, and M.~Takeda.
\newblock {\em Dirichlet forms and symmetric {M}arkov processes}, volume~19 of
  {\em De Gruyter Studies in Mathematics}.
\newblock Walter de Gruyter \& Co., Berlin, extended edition, 2011.

\bibitem{galaskub}
E.~I. Galakhov and A.~L. Skubachevskii.
\newblock On {F}eller semigroups generated by elliptic operators with
  integro-differential boundary conditions.
\newblock {\em J. Differential Equations}, 176(2):315--355, 2001.

\bibitem{hansen2023positive}
W.~Hansen and K.~Bogdan.
\newblock Positive harmonically bounded solutions for semi-linear equations,
  2023.
\newblock arXiv:2212.13999.

\bibitem{MR202197}
N.~Ikeda, M.~Nagasawa, and S.~Watanabe.
\newblock A construction of {M}arkov processes by piecing out.
\newblock {\em Proc. Japan Acad.}, 42:370--375, 1966.

\bibitem{MR232439}
N.~Ikeda, M.~Nagasawa, and S.~Watanabe.
\newblock Branching {M}arkov processes. {I}.
\newblock {\em J. Math. Kyoto Univ.}, 8:233--278, 1968.

\bibitem{MR238401}
N.~Ikeda, M.~Nagasawa, and S.~Watanabe.
\newblock Branching {M}arkov processes. {II}.
\newblock {\em J. Math. Kyoto Univ.}, 8:365--410, 1968.

\bibitem{MR246376}
N.~Ikeda, M.~Nagasawa, and S.~Watanabe.
\newblock Branching {M}arkov processes. {III}.
\newblock {\em J. Math. Kyoto Univ.}, 9:95--160, 1969.

\bibitem{MR270456}
N.~Ikeda, M.~Nagasawa, and S.~Watanabe.
\newblock Correction to ``{B}ranching {M}arkov processes. {II}''.
\newblock {\em J. Math. Kyoto Univ.}, 11:195--196, 1971.

\bibitem{MR4226142}
O.~Kallenberg.
\newblock {\em Foundations of modern probability}, volume~99 of {\em
  Probability Theory and Stochastic Modelling}.
\newblock Springer, Cham, 2021.
\newblock Third edition.

\bibitem{kulik15}
A.~Kulik.
\newblock {\em Introduction to ergodic rates for {M}arkov chains and
  processes}, volume~2 of {\em Lectures in Pure and Applied Mathematics}.
\newblock Potsdam University Press, Potsdam, 2015.
\newblock With applications to limit theorems.

\bibitem{kunze20}
M.~Kunze.
\newblock Diffusion with nonlocal {D}irichlet boundary conditions on unbounded
  domains.
\newblock {\em Studia Math.}, 253(1):1--38, 2020.

\bibitem{MR0415784}
P.-A. Meyer.
\newblock Renaissance, recollements, mélanges, ralentissement de processus de
  markov.
\newblock {\em Annales de l'institut Fourier}, 25(3-4):465--497, 1975.

\bibitem{MR3644418}
R.~L. Schilling.
\newblock {\em Measures, integrals and martingales}.
\newblock Cambridge University Press, Cambridge, second edition, 2017.

\bibitem{MR958914}
M.~Sharpe.
\newblock {\em General theory of {M}arkov processes}, volume 133 of {\em Pure
  and Applied Mathematics}.
\newblock Academic Press, Inc., Boston, MA, 1988.

\bibitem{MR3308364}
K.~Taira.
\newblock {\em Semigroups, boundary value problems and {M}arkov processes}.
\newblock Springer Monographs in Mathematics. Springer, Heidelberg, second
  edition, 2014.

\bibitem{MR4247975}
F.~Werner.
\newblock Concatenation and pasting of right processes.
\newblock {\em Electron. J. Probab.}, 26:Paper No. 50, 21, 2021.

\end{thebibliography}
\end{document}